\documentclass[11pt]{amsart}
\usepackage{amssymb, amsmath, amsthm}
\usepackage{amsrefs}
\usepackage{graphicx}
\usepackage[final]{hyperref}
\usepackage[a4paper, centering]{geometry}
\usepackage{color}
\usepackage{graphicx} 

\newtheorem{theorem}{Theorem}
\newtheorem{proposition}[theorem]{Proposition}
\newtheorem{lemma}[theorem]{Lemma}

\theoremstyle{definition}

\theoremstyle{remark}
\newtheorem{remark}[theorem]{Remark}
\newtheorem{example}[theorem]{Example}
\def\R{\mathbb{R}}

\def\N{\mathbb{N}}

\def\inradius{R_\Omega}

\def \Dinf{\Delta _\infty} 
\def \L{\Lambda}
\def \l {\lambda} 
\def\Lbda{\Lambda'}
\def\Lbdb{\Lambda^*}
\def\r{r^*}

\def\JLpa{J_{\Lambda} ^{p, \lambda}}
\def\JLpja{J_{\Lambda} ^{p_j,  \lambda}}
\def\JLia{J_{\Lambda}^{\lambda}}
\def\upa{u^{p,\lambda}}
\def\upja{u^{p_j, \lambda}}

\newcommand{\Lip}[1]{{\left\| \nabla #1 \right\|}_\infty}

\DeclareMathOperator{\argmin}{argmin}
\DeclareMathOperator{\llipst}{Lip}
\DeclareMathOperator{\dist}{dist}

\begin{document}

 %amsart format
\title[On the supremal version of the Alt--Caffarelli minimization  problem]%
{On the supremal version \\ of the Alt--Caffarelli minimization  problem }%
\author[G.~Crasta, I.~Fragal\`a]{Graziano Crasta,  Ilaria Fragal\`a}
\address[Graziano Crasta]{Dipartimento di Matematica ``G.\ Castelnuovo'', Univ.\ di Roma I\\
P.le A.\ Moro 2 -- 00185 Roma (Italy)}
\email{crasta@mat.uniroma1.it}

\address[Ilaria Fragal\`a]{
Dipartimento di Matematica, Politecnico\\
Piazza Leonardo da Vinci, 32 --20133 Milano (Italy)
}
\email{ilaria.fragala@polimi.it}

 	%35R35   	Free boundary problems
 	%49J45   	Methods involving semicontinuity and convergence; relaxation
\date{November 30, 2018}

\begin{abstract}
This is a companion paper to our recent work \cite{CF8}, where we studied the 
interior Bernoulli free boundary for the infinity Laplacian.  Here we consider its variational side, 
which corresponds to the supremal version of the Alt--Caffarelli minimization problem.  
\end{abstract}

\maketitle

\medskip
\section{Introduction}
Given a nonempty open bounded domain $\Omega \subset \R ^n $, $n\geq 2$, 
we consider the free boundary problem  
\begin{equation}\label{f:P}
m _\Lambda := \min \Big \{ J _\Lambda (u) := \Lip{u} + \Lambda | \{ u > 0 \} | \ :\ 
u\in\llipst_1(\Omega)
 \Big \} \,.\tag*{$(P)_\Lambda$}  
\end{equation}
where $\Lambda$ is a positive constant, $|\{u>0 \}|$ denotes the Lebesgue measure of the set $ \{x \in \Omega \ :\ u (x) >0 \}$, and 
\begin{equation}\label{f:lip1}
\llipst_1(\Omega):= \left\{
u\in %C(\overline{\Omega}) \cap 
W^{1,\infty}(\Omega):\
u\geq 0\ \text{in}\ \Omega,\
u=1\ \text{on}\ \partial\Omega
\right\}\,.
\end{equation}

This may be viewed as the supremal version of the Alt--Caffarelli minimization problem, 
which for $p$-growth energies reads
\begin{equation}\label{f:VPp}
\min \Big \{ \int _\Omega { |\nabla u|}  ^ p \, dx +  \Lambda | \{ u > 0 \} | \ :\ 
u\in W ^ {1, p }(\Omega)\, , \ u = 1 \hbox{ on } \partial \Omega \Big \}\,.
\end{equation}

In both the minimization problems \ref{f:P} and \eqref{f:VPp},  the free boundary is given by the set 
$$F (u):= \partial \{u > 0\} \cap \Omega \,.$$   Clearly, the free boundary will be not empty only if the parameter $\Lambda$ is taken sufficiently large, in order that
the measure term
becomes active in the competition between the two addenda  in the energy functional. 
To make the difference is 
the gradient term:  the integral functional appearing in \eqref{f:VPp}  is converted into the supremal functional $\Lip{u}$ in  \ref{f:P} . 

Problem \eqref{f:VPp} has a long history:  
starting from the groundbreaking paper \cite{AlCa}, where it was introduced in the linear case $p=2$,  it
 has been widely studied in later works for any $p \in (1, + \infty)$ (see for instance \cite{BS09, DanPet, DaKa2010, KawSha}). 
In particular, the topic which has been object of a thorough investigation is the  regularity of the free boundary, which has been settled to be 
locally analytic except for a $\mathcal H ^ {n-1}$-negligible singular set \cite{AlCa, DanPet}. 

Among the motivations behind problem \eqref{f:VPp}, of chief importance is that its minimizers  solve for a suitable constant $c>0$ the `Bernoulli problem for the $p$-laplacian', namely
\begin{equation}\label{f:Pp}
\begin{cases}
\Delta _p u = 0 & \text{ in }  \{u > 0\} \cap \Omega,
\\
u = 1 & \text{ on } \partial \Omega,
\\ 
|\nabla u| = c   
& \text{ on }  \partial \{u > 0\} \cap \Omega. 
\end{cases}  
\end{equation}
This overdetermined system, which is named  after Daniel Bernoulli (in particular from his law in hydrodynamics),  has 
many physical and industrial   
applications, not only in fluid dynamics, but also in other contexts, such as optimal insulation  and electro-chemical machining.  For a more precise description of the applied side, including several references, see \cite[Section 2]{FR97}.

In our recent paper \cite{CF8}, we have studied  the
existence  and uniqueness of solutions to the  following
`Bernoulli problem for the $\infty$-laplacian'
\begin{equation}\label{f:Pinfty}
\begin{cases}
\Delta _\infty u = 0 & \text{ in }\{u > 0\} \cap \Omega,
\\
u = 1 & \text{ on } \partial \Omega,
\\ 
|\nabla u| = \lambda    
& \text{ on } \partial \{u > 0\} \cap \Omega\,. 
\end{cases}  
\end{equation}
Recall that the $\infty$-laplacian is the degenerated nonlinear operator  defined  by
$$
\Delta _\infty u := \nabla ^ 2 u \nabla u \cdot \nabla u\qquad \forall u \in C ^ 2 (\Omega)\,.
$$
It is well-known since
Bhattacharya, DiBenedetto and Manfredi \cite{BDM} that 
solutions to 
\begin{equation}\label{f:inftyharm}
\Delta _\infty u = 0\,, 
\end{equation}
must be intended in the viscosity sense, since the differential operator  is not in divergence form. 
Moreover, it is a widely recognized fact that \eqref{f:inftyharm} may be seen 
as the fundamental PDE of Calculus of Variations in $L ^ \infty$, {\it i.e.}, an analogue of the Euler-Lagrange equation when considering variational problems for supremal functionals, the simplest of which is the $L ^ \infty$ norm of the gradient.  Indeed, as proved by Jensen \cite{Jen}, a function $u$ is a viscosity solution to  \eqref{f:inftyharm} in an open set $A$ under a Dirichlet boundary datum $g \in C ^ 0 (\partial A)$  if and only if it is an absolutely minimizing Lipschitz extension of  $g$ (this property, usually shortened as AML, means that  $u = g$ on $\partial A$ and, 
for every $U \Subset A$,  $u$ minimizes the $L ^ \infty$ norm of the gradient on $U$ among functions $v$ which agree with $u$ on $\partial U$). 
After Jensen, variational problems for supremal functionals have been object of several works,  among which we quote with no attempt of completeness  \cite{Barron, BaJeWa, ACJ, Cran}. 

In this perspective, 
the aim of the present work is to investigate the variational side of problem 
\eqref{f:Pinfty}, which is precisely the new free boundary problem of  supremal type \ref{f:P}. 

Our results are presented in the next section, which is divided into three parts:

\medskip
$\bullet$ 
In Section \ref{s:analysis}  we study the existence and uniqueness of non-constant solutions to \ref{f:P} on convex domains (see Theorems \ref{t:Linfty} and \ref{t:uniqP}). 
In particular,  we introduce  the `variational 
$\infty$-Bernoulli constant' 
\begin{eqnarray}
& \Lambda _{\Omega, \infty}   := \inf  \Big \{ \Lambda >0 \ :\ (P)_\L \text{ admits a non-constant solution }\Big \} \, ,  & \label{f:L} 
 \end{eqnarray}
 and we give a geometric characterization of it,   involving the family of parallel sets of $\Omega$. 
 This formula is obtained with the help of the Brunn-Minkowski inequality,
and allows 
to compute
explicitly the value of $\Lambda _{\Omega, \infty}$, at least for simple geometries. In particular, if  we compare $ \Lambda _{\Omega, \infty} $ with  the 
`$\infty$-Bernoulli constant',  defined in \cite{CF8} by
\[
\lambda_{\Omega, \infty} := \inf\{\lambda > 0:\ \text{\eqref{f:Pinfty} admits a non-constant solution}\}\, ,  
\]
it turns out that $\Lambda  _{\Omega, \infty} \geq \lambda _{\Omega, \infty}$; the computation on balls reveals that the inequality can be strict. 
 
 Still using its geometric characterization, we prove  that  $\Lambda _{\Omega, \infty} $  satisfies an isoperimetric inequality, namely that it is minimal on balls under a volume constraint (see Theorem \ref{t:isoper}).  This was inspired by a result in the same vein 
 by Daners and Kawohl, for the variational $p$-Bernoulli constants associated with problem \eqref{f:VPp} (see \cite{DaKa2010}).

\medskip
$\bullet$  In Section \ref{s:Bernoulli} we elucidate the relationship between solutions to problem \ref{f:P} and  solutions to the Bernoulli problem \eqref{f:Pinfty} for the $\infty$-Laplacian,  by showing they are closely related to each other. More precisely we answer the following natural questions
(see Propositions \ref{p:PP} and \ref{p:varsol}):

\begin{itemize}
\item[$Q_1$.]  Let $\L \geq \L _{\Omega,  \infty}$, so that problem \ref{f:P} admits a non-constant solution.  

\hskip -.4cm Does a solution to problem \ref{f:P}  solve problem  \eqref{f:Pinfty} for some $\lambda$? 

\smallskip
\item[$Q_2$.] Let $\l \geq \l _ {\Omega, \infty}$, so that problem \eqref{f:Pinfty} admits a non-constant solution. 

\hskip -.4cm Does a solution to problem \eqref{f:Pinfty}  solve problem \ref{f:P}  for some $\Lambda$?  
\end{itemize}

\medskip
$\bullet$  In Section \ref{s:approx} we show that problem \ref{f:P} 
 can be obtained not merely in a heuristic way, but 
by performing a rigorous passage to the limit as $p \to + \infty$ 
in a family of minimum problems of the Alt--Caffarelli type for energies with $p$-growth (see Theorem \ref{t:convmin}).  This result, which is somehow highly expected, is 
not straightforward. In fact, the passage to the limit as $p \to + \infty$  in problems of type \eqref{f:VPp} (suitably rescaled) provides a minimization problem with gradient constraint, similarly as it was shown by Kawohl and Shahgholian for {\it exterior} Bernoulli problems, see \cite{KawSha}. In order to  arrive at problem \ref{f:P}, one needs to perform a further minimization, with respect to an additional parameter which plays the role of a multiplier for the gradient constraint.

\medskip
The proofs of the results described above are given respectively in Sections \ref{sec:proof1}, \ref{sec:proof2}, and \ref{sec:proof3}.

\section{Results}

\subsection{Analysis of  problem \ref{f:P}} \label{s:analysis} 
A major role in our study is played by the parallel sets of $\Omega$ and by the functions $v_r$ defined  for every $r \in [0, R _\Omega]$  
(with $R _\Omega$:= inradius of $\Omega$)
respectively by
\[
\Omega_r := \{x\in\Omega:\ \dist(x, \partial\Omega) > r\}
\]

and
\begin{equation}\label{f:ur}
v_r(x) := \left[1 - \frac{1}{r} \dist(x, \partial\Omega)\right]_+,
\qquad x\in\overline{\Omega}\,.
\end{equation}
Clearly, for every $r\in (0, \inradius]$, we have 
\[
v_r\in\llipst_1(\Omega)\, , \quad \Lip{v_r} = 1/r\, , \quad
\{v_r > 0\}\cap \Omega = D_r := \Omega\setminus\overline{\Omega_r}\,,
\]
where $\llipst_1(\Omega)$ is the set of functions defined
in \eqref{f:lip1}.
(To be precise, we consider any function $u\in\llipst_1(\Omega)$
as an H\"older continuous function in $\overline{\Omega}$,
since $u-1\in W^{1,p}_0(\Omega)$ for every $p>1$.)

We are going to provide an explicit characterization for the variational infinity Bernoulli constant introduced in  \eqref{f:L}
and to show that, for $\Lambda \geq \Lambda _{\Omega, \infty}$,  non-contant solutions are precisely of the form \eqref{f:ur}.  

Throughout the paper, we assume with no further mention that  
\[
\text{$\Omega$ is an open bounded  \textit{convex} subset of $\R^n$, $n \geq 2$. }
\]
Moreover, since $\Omega$ will be fixed, for simplicity in the sequel we simply write $\Lambda _\infty$  in place of 
$\Lambda _{\Omega, \infty}$.

\begin{theorem}[Identification of $\Lambda _\infty$]\label{t:Linfty}
There exists a unique value  $r^*$ in the interval $(0, R _\Omega)$ such that 
\begin{equation}\label{f:r*}
\frac{|\partial\Omega_{\r}|}{|\Omega_{\r}|} = \frac{1}{\r}\,, 
\end{equation}
and the variational infinity Bernoulli constant $\Lambda _\infty$ defined in \eqref{f:L} agrees with $$\Lbdb
:= 
\frac{1}{\r |\Omega_{\r}|} \, .$$
More precisely,  we have

\smallskip
	\begin{itemize}
		\item[--]    for $\L < \Lbdb $, $(P)_\L$ admits uniquely the constant solution $u \equiv 1$; 
		\medskip
		\item[--] for $\L  = \Lbdb $, $(P)_\L$ admits the constant solution $u \equiv 1$ and the non-constant solution $v _ {r^*}$; 
		\medskip
		\item[--] for $\L > \Lbdb$, $(P)_\L$ does not admit the constant solution $u \equiv 1$ and admits the non-constant solution $v _{r_\Lambda}$, 
		where $r _\L$ is the smallest root  of the equation
$\frac{1}{r^2 |\partial\Omega_r| } = \Lambda $.
		
				\end{itemize}  
Moreover, for $\Lambda \geq \Lbdb$,
any non-constant solution $v$ to \ref{f:P} has the same Lipschitz constant and the same positivity set, namely satisfies
\begin{equation}\label{f:split} 
\Lip{v} = 1/r_\Lambda \qquad \text{ and  } \qquad \{v > 0\} = D_{r_\Lambda}\,.
\end{equation} 
(being $r _{\Lambda ^*}= r^*$).
\end{theorem}

\begin{remark}[Computation of $m _\L$]
The above result implies in particular that  the infimum $m _\L$   of problem $(P)_\L$  can be explicitly computed as
$$m _\L = \begin{cases}
\L |\Omega| & \text{ if } \L \leq \L _\infty
\\
 \frac{1}{r_\L} + \L |D_{r_\L}| & \text{ if } \L \geq \L _\infty
 \,.
 \end{cases}$$  
\end{remark}

\begin{remark}[Infinity harmonic solution to \ref{f:P}]\label{c:corharm}
For every $\L \geq \L _\infty$, among solutions to  \ref{f:P}, there is exactly one which is infinity harmonic in its positivity set, 
namely the infinity harmonic potential $w_{r_\L}$ of $D_{r_\Lambda}$, defined as the unique solution to the Dirichlet boundary value problem 
\begin{equation}\label{f:pot} 
\begin{cases}
\Dinf w_{r_\L} = 0 & \text{ in } D_{r_\Lambda}
\\
w_{r_\L}= 1 & \text{ on } \partial \Omega
\\ 
w_{r_\L} = 0 & \text{ on } \partial \Omega_{r_\Lambda}\,.
\end{cases} 
\end{equation}
Indeed, if \ref{f:P} admits an infinity harmonic solution, it agrees necessarily with $w_{r_\L}$ because its positivity set is uniquely determined by \eqref{f:split}. To see that $w_{r_\L}$ is actually a solution of \ref{f:P}, we observe that  it has the same positivity set as $v_{r_\L}$, and a
Lipschitz constant not larger than $v_{r_\L}$ (because $w_{r _\L}$  has the AML property mentioned in the Introduction). 
Hence we have necessarily $J_\L (v_{r_\L}) = J_\L (w_{r_\L})$, yielding optimality. 
\end{remark}

\medskip

Below we give the explicit expression of $\Lambda _\infty$ for some simple geometries (the ball in any space dimension, and  the square in dimension $2$), 
and we show that it enjoys a nice isoperimetric property.  

\begin{example}[Ball]\label{e:ball}
Let $\Omega = B_R \subset \R^n$
be the $n$-dimensional open ball of radius $R$
centered at the origin. Then $\inradius = R$ and
\[
\psi(r) := \frac{|\partial B_{R-r}|}{|B_{R-r}|} = \frac{n}{R-r}\,,
\qquad
r\in [0,R).
\]
The unique solution to the equation $\psi(r) = 1/r$ is $
\r = \frac{R}{n+1}$, 
hence
\[
\Lambda_\infty (B_R) = \frac{1}{\r |\Omega_{\r}|} = 
\frac{1}{\kappa_n \r (R-\r)^{n}} = \frac{(n+1)^{n+1}}{\kappa_n n^n R^{n+1}}
\qquad
(\kappa_n := |B_1|)\,.
\]
In particular, for $n=2$ we get
\[
\Lambda_\infty (B _R) = \frac{27 }{  4\pi R^3}\,. 
\]
\end{example}

\begin{example}[Rectangle]
Let $n=2$ and $\Omega = Q_{a, b}:= (0,a)\times (0,b)$, with $0 < b \leq a$.
In this case $\inradius = b/2$ and
\[
\psi(r) :=   \frac{|\partial Q_{a-r, b-r}|}{|Q_{a-r, b-r}|} = 2 \frac{a+b-4r}{(a-2r)(b-2r)}\,,
\qquad r\in [0, b/2).
\]
The unique solution to the equation $\psi(r) = 1/r$  in the interval 
$[0, b/2)$ is
\[
\r = \frac{a+b-\sqrt{a^2+b^2-ab}}{6}
\]
In particular, for the square $Q_a:= Q_{a, a}$  we get $
\r = \frac{a}{6}$, and hence

$$
\Lambda_\infty (Q_a)= \frac{27}{2 a^3}\,.
$$ 
\end{example}

 \begin{theorem}[Isoperimetric inequality]\label{t:isoper}
Denoting by $\Omega^*$ a ball
with the same volume as $\Omega$, it holds 
$$\Lambda_\infty(\Omega) \geq \Lambda_\infty(\Omega^*)\,,$$
with  equality sign  if and only if $\Omega$ is a ball.
\end{theorem}

\begin{remark}\label{r:chiarimento1}
As mentioned in the Introduction, a result analogous  to Theorem \ref{t:isoper} has been obtained in \cite{DaKa2010} for a variational Bernoulli constant 
related to the Alt--Caffarelli minimization problems for $p$-growth energies. 
We wish to point out that Theorem \ref{t:isoper} cannot be obtained simply by passing to the limit as $p \to + \infty$ in the isoperimetric inequality by Daners and Kawohl. 
After the discussion in Section  \ref{s:approx}, we will be in a position to
give  more  details in this respect (see 
Remark \ref{r:chiarimento2}).  \end{remark}

\bigskip 

We now turn our attention to the uniqueness of solutions to problem \ref{f:P}. To that aim, we introduce the  {\it singular radius of $\Omega$}, defined as  
\begin{equation}\label{f:Lsing}
r_{sing}:= \dist (\Sigma, \partial \Omega)\, , 
\end{equation}
where $\Sigma$ denotes the cut locus of $\Omega$ (i.e., the closure of the set of points where the distance from $\partial \Omega$ is not differentiable). 

Accordingly, since the map $\mathcal R: \L \mapsto r_\L$ (with $r _\L$ defined as in Theorem \ref{t:Linfty}) turns out to be monotone decreasing from $[\L ^*, + \infty)$ to $[r^* , 0)$ (see Lemma \ref{l:fl} below), the value
\begin{equation}\label{f:rsing}
\L_{sing}:= \mathcal R ^ {-1} (r_{sing})\, , 
\end{equation}
 is uniquely defined. 

We point out that the radius $r _{sing}$ may be smaller or larger than $r ^*$ according to the domain under consideration (and consequently $\L _{sing}$ may be larger or smaller than $\L _\infty = \L ^*$). 
For instance in dimension $n=2$,  if $\Omega = B _R$ (the ball of radius $R$), it holds
$$r_{sing} = R > r ^* = \frac{R}{3}\,;$$ 
on the other hand, if $\Omega= Q_a$  (the square of side $a$), we have
$$r_{sing} = 0 < r ^* = \frac{a}{6}\,.$$

We are now in a position to discuss the uniqueness of solutions to \ref{f:P}. 

\begin{theorem}[Uniqueness  threshold]\label{t:uniqP}
We have:
\begin{itemize}
\item[--] If $\L_{sing} \leq\L _\infty$,  \ref{f:P}  admits a unique solution (given by  $v_{r_\Lambda}$)  if and only if $\L > \L _\infty$. 

\item[--] If $\L_{sing} >\L _\infty$,  \ref{f:P}  admits a unique solution (given by  $v_{r_\Lambda}$) 
if and only if $\L \geq \L _{sing}$. 
\end{itemize} 
\end{theorem}

\bigskip
\subsection {Relationship with the $\infty$-Bernoulli problem}\label{s:Bernoulli}
We now examine the link between solutions to \ref{f:P} and solutions to \eqref{f:Pinfty}. 
A first glance in this direction has been already given in Remark  \ref{c:corharm}. 
A more precise answer to the 
questions $Q_1$ and $Q_2$  stated in the Introduction is contained in the next two statements. 

\begin{proposition}[Solutions to  \ref{f:P} versus solutions to  \eqref{f:Pinfty}]\label{p:PP} 

Let $\L \geq \L _ \infty$. 

 Among solutions to problem \ref{f:P}, there is exactly one which solves \eqref{f:Pinfty}  (for  $\l =   \frac{1}{r_\L}$): 
it is given by the infinity harmonic potential $w_{r_\L}$ of $D_{r _\Lambda}$.  

\smallskip
 
In particular, when \ref{f:P} admits a unique solution, this one solves \eqref{f:Pinfty}  (for  $\l =   \frac{1}{r_\L}$), 
and it is given by $w_{r_\L} = v _{r_\L}$.  
\end{proposition}

\begin{proposition}[Solutions to  \eqref{f:Pinfty} versus solutions to  \ref{f:P}]\label{p:varsol}  
Let $\l \geq R _\Omega ^ {-1}$. 

Among solutions to problem \eqref{f:Pinfty}, there is one which solves \ref{f:P} 
if and only if
\begin{equation}\label{f:cond1}
\lambda\geq \frac{1}{r^*}\ .
\end{equation}
In this case such solution is given by $w_{r_\L}$, with $r_\Lambda  = \frac{1}{\lambda}$, 
and agrees with $v _{r_\L}$ if and only if  
\[%\begin{equation}\label{f:cond2} 
\lambda\geq \max \Big \{ \frac{1}{r^*}, \frac{1}{r_{sing}}\Big \}  \,.
\]
 \end{proposition}

\subsection {Approximation by minima of $p$-energies}\label{s:approx}
We now show that problem \ref{f:P}
arises through an asymptotic analysis as $p \to + \infty$
of problems of the type \eqref{f:VPp}.

Let $\Lambda  \geq \Lambda _\infty$ be fixed.  
For every $\lambda>0$ and $p>1$,  let us
consider  the minimization problem 
\begin{equation}\label{f:minp}
\min\left\{\JLpa (u):\ u\in W_1^{1,p}(\Omega)\right\}\,,
\end{equation}
where  the functionals  $\JLpa$ are defined by 
$$ \JLpa (u) :=
\frac{1}{p}\int_\Omega \left(\frac{|\nabla u|}{\lambda}\right)^p\, dx + \lambda +
\frac{p-1}{p} \Lambda |\{u>0\}|,
\qquad\forall u \in W_1^{1,p}(\Omega) := 1 + W ^ {1,p}_0 (\Omega)\,.
$$
Incidentally, let us  recall from \cite{DanPet} that a  solution to problem \eqref{f:minp}
exists,  is non negative in $\Omega$ and satisfies the overdetermined boundary value problem   \eqref{f:Pp} with $c = \lambda \Lambda^{1/p}$ 
(provided the Neumann boundary condition on the free boundary is intended in a suitable weak sense).

If, for a given $\lambda >0$, we consider 
a sequence of solutions  $(u^{p_j, \lambda})_j$  to problem \eqref{f:minp} for 
$p = p _j \to + \infty$, it is not difficult to see that, 
up to passing to a (not relabeled) subsequence, the functions 
$u^{p_j, \lambda}$ converge 
uniformly in $\overline{\Omega}$  to a function $ u^\lambda$
which is infinity harmonic in its positivity set and solves the variational problem
\begin{equation}\label{f:mina} 
\min \Big \{ \JLia (u)  \ :\ 
u\in\llipst_1(\Omega)
 \Big \} \,,
\end{equation}

the functionals $J_{\Lambda}^ \lambda$ being defined 
on $\llipst_1(\Omega)$ by 
\begin{equation}\label{f:double}
\JLia(u) :=
\begin{cases}
\lambda + \Lambda  |\{u>0\}|& \text{  if}\ \Lip{u} \leq \lambda \\
\smallskip
+\infty &\text{   otherwise}.
\end{cases}
\end{equation}
The proof of this fact, which will be detailed in Section \ref{sec:proof3} (see Lemma \ref{l:conv}), is similar to the one given by Kawohl and Shahgholian in the paper \cite{KawSha}, where they deal with the asymptotics  as $p  \to + \infty$ of exterior $p$-Bernoulli problems.  (Indeed, by computing a $\Gamma$-limit, they arrive precisely at a minimum problem with gradient constraint analogue to \eqref{f:mina}, in which they fix $\Lambda = 1$.)

Now we observe that the functionals $\JLia$  are related to $J _\L$ by the equality
\begin{equation}\label{f:ce}
J_\Lambda(u) = \inf_{\lambda > 0} \JLia (u),
\qquad\forall u \in \llipst_1(\Omega) 
\end{equation}
(and actually the above infimum can be equivalently taken over $\lambda \geq 1/\inradius$, since for every $u\in \llipst_1(\Omega)$ with $\Lip{u} < 1/\inradius$ it holds
$|\{ u > 0 \}| = |\Omega|$).
 
The validity of \eqref{f:ce} is precisely the reason why we put the constant addendum $\lambda$ into the expression of the functionals $\JLpa$: it can be interpreted as a sort of Lagrange multiplier for the gradient constraint  $\Lip{u} \leq \lambda$ which appears when passing to the limit  at fixed $\lambda$.

In the light of \eqref{f:ce}, it is now natural to guess how the value of $m _\L$ can be obtained from $\JLpa$ in the limit as $p \to + \infty$:  one has to take a double infimum, over $u\in W ^ {1, p} _1 (\Omega)$ and over $\lambda \geq 1/ \inradius$. 

\begin{theorem}[$p$-approximation]\label{t:convmin}

Let $\Lambda\geq\L_\infty$. Then
\begin{equation}\label{f:min}
\lim_{p\to +\infty}
\inf_{\lambda \geq 1/\inradius}
\inf_{u\in W_1^{1,p}(\Omega)} \JLpa(u)
= m _\Lambda \,.
\end{equation}
\end{theorem}

\begin{remark}\label{r:chiarimento2}
The
variational $p$-Bernoulli constant $\Lambda_{p}$ considered by Kawohl and Daners in \cite{DaKa2010} 
agrees with the infimum of positive $\lambda$ such that problem \eqref{f:minp} (with $\Lambda = 1$) admits a non-constant solution. 
In the limit as $p \to + \infty$, $\Lambda _p$ 
does not converge to $\Lambda _\infty$ (in fact, in \cite{CF8} we proved that  $\lim _{p \to + \infty} \Lambda _p = 1/ \inradius$), 
and this is precisely the reason why, as mentioned in Remark \ref{r:chiarimento1}, 
Theorem \ref{t:isoper} cannot be obtained by passing to the limit as $p \to + \infty$ in the isoperimetric inequality proved by Daners-Kawohl in \cite{DaKa2010}. 
In view of Theorem \ref{t:convmin}, one should  not be surprised by the missed convergence of $\Lambda_p$ to $\Lambda _\infty$. Indeed,  \eqref{f:min} suggests that, 
in order to find a $p$-approximation of $\Lambda _\infty$, one should consider rather the constants
$\widetilde \Lambda _p$ defined as the infimum of positive $\Lambda$ such that problem 
\[
\inf_{u\in W_1^{1,p}(\Omega)}    \inf_{\lambda \geq 1/\inradius}
\JLpa(u)
\]  
admits a non-constant solution.  Indeed a straightforward formal computation shows that
\[
\inf_{\lambda \geq 1/\inradius}
\JLpa(u) = \frac{p+1}{p} \| \nabla u\| _{L ^p(\Omega)} ^ {p/(p+1)} + \frac{p-1}{p} \Lambda |\{ u > 0 \} | \to J _\Lambda (u) \qquad \text{ as } p \to + \infty\,.
\] 
\end{remark}

\section{Proof of the results in Section \ref{s:analysis}}\label{sec:proof1}

To start, we establish a simple existence result:

\begin{lemma}[Existence of solutions to \ref{f:P}]\label{l:existence}
Problem \ref{f:P} admits a solution for every $\Lambda > 0$.
\end{lemma}

\begin{proof}
Let $\Lambda > 0$ and let
$(u_j)\subset\llipst_1(\Omega)$ be a minimizing sequence for $J_\Lambda$.
Clearly, it is not restrictive to assume that
$J_\Lambda(u_j) \leq J_\Lambda(1) = \Lambda |\Omega|$,
so that $\| \nabla u_j \| _\infty \leq \Lambda |\Omega|$ for every $n$.
Hence, the sequence $(u_j)$ is equi-Lipschitz. 
Since $u_j = 1$ on $\partial\Omega$ for every $n$,
by Ascoli--Arzel\'a's theorem we deduce that there exist a subsequence (non relabeled)
and a function $u\in\llipst_1(\Omega)$ such that $u_j\to u$ uniformly in $\overline{\Omega}$.

Since
\[
\forall j\in\N:\quad
|u_j(x) - u_j(y)| \leq \Lip{u_j}\, |x-y|
\qquad \forall x,y\in\overline{\Omega},
\] 
passing to the limit we get
\[
|u(x) - u(y)| \leq \left(\liminf_{n\to+\infty}\Lip{u_j}\right)\, |x-y|
\qquad \forall x,y\in\overline{\Omega},
\]
so that $\Lip{u} \leq \liminf_j \Lip{u_j}$.

Moreover, since $u_j(x) \to u(x)$ for every $x\in\overline{\Omega}$,
it is easy to verify that
\[
\chi_{\{u > 0\}} (x) \leq
\liminf_{j\to +\infty}
\chi_{\{u_j > 0\}} (x),
\]
hence, by Fatou's Lemma, we deduce that
\[
|\{u > 0\}|
= \int_\Omega \chi_{\{u > 0\}}
\leq
\liminf_{j\to +\infty}
\int_\Omega \chi_{\{u_j > 0\}}
= \liminf_{j\to +\infty} |\{u_j > 0\}|.
\]
In conclusion, $J_\Lambda$ is lower semicontinuous with respect to the
uniform convergence, hence $u$ is a minimum point for $J_\Lambda$.
\end{proof}

We establish now the first part in the statement of Theorem \ref{t:Linfty}:

\begin{lemma}\label{l:rinfty} 
There exists a unique value $\r$ in the interval $(0, R _\Omega)$ satisfying equality
\eqref{f:r*}. Moreover, 
the function $r \mapsto r |\Omega_r|$ 
attains its unique maximum on $[0, \inradius]$ precisely at $\r$: 
$$\max _{[0, \inradius]} \big ( r |\Omega_r| \big ) = \r |\Omega _{\r}|\,.$$

\end{lemma}

\bigskip
\begin{remark}
In the proof of Lemma \ref{l:rinfty} below (and also of the subsequent Lemma \ref{l:fl}), 
we make  heavily use of the Brunn-Minkowski inequality for the volume functional of the parallel sets $\Omega _r$, and for their surface area measure as well. This motivates the convexity assumption made on the domain $\Omega$.
\end{remark}

{\it Proof of Lemma \ref{l:rinfty}}. 
Let us prove  the following claim:
\[
\psi(r) := \frac{|\partial\Omega_r|}{|\Omega_r|}
\quad r\in [0, \inradius),\ \text{is continuous, increasing and}\
\lim_{r\to\inradius-} \psi(r) = +\infty.
\]
Namely, by the Brunn--Minkowski inequality the functions
$r\mapsto |\Omega_r|^{1/n}$,
$r\mapsto |\partial\Omega_r|^{1/(n-1)}$ are concave in $[0, \inradius]$.
(Here and in the following, $|\partial\Omega_r|$ denotes the $(n-1)$-dimensional
Hausdorff measure of $\partial\Omega_r$.)
Hence, $\psi$ is continuous in $[0,\inradius)$, and the composition
$r \mapsto \log |\Omega_r| = n \log |\Omega_r|^{1/n}$
is concave in 
the same interval. In particular, 
since $\frac{d}{dr} |\Omega_r| = -|\partial\Omega_r|$, we conclude that
\[
\psi(r) = -\frac{d}{dr} \log |\Omega_r|
\]
is increasing in $[0,\inradius)$.
Finally, from the isoperimetric inequality we have 
\[
\psi(r)^n = \frac{|\partial\Omega_r|^n}{|\Omega_r|^n}
= \frac{|\partial\Omega_r|^n}{|\Omega_r|^{n-1}} \cdot \frac{1}{|\Omega_r|}
\geq \frac{|\partial B_1|^n}{|B_1|^{n-1}} \cdot \frac{1}{|\Omega_r|}
\to +\infty,
\quad r\to \inradius^-.
\]
The claim follows.  As a consequence, there exists a unique value $\r\in (0, R _\Omega)$ such that 
\begin{equation}\label{f:cfr}
\frac{|\partial\Omega_r|}{|\Omega_r|} < \frac{1}{r}
\quad\forall r\in(0,\r),\qquad
\frac{|\partial\Omega_r|}{|\Omega_r|} > \frac{1}{r}
\quad\forall r\in (\r, \inradius)\,.
\end{equation}

Finally, in order to determine the maximum of the function $\varphi(r):=r |\Omega_r|$  
on $[0, \inradius]$,  we compute its first derivative
\[
\varphi'(r) = |\Omega_r| - r |\partial\Omega_r| = 
\varphi(r) \left(\frac{1}{r} - \frac{|\partial\Omega_r|}{|\Omega_r|}\right).
\] 
By \eqref{f:cfr}, it holds
$\varphi'(r) > 0$ for $r\in (0, \r)$
and $\varphi'(r) < 0$ for $r\in (\r, \inradius)$,
hence $\varphi$ attains its strict maximum at $\r$.
\qed

\medskip
In the next lemma, which is the key step towards the completion of the proof of Theorem~\ref{t:Linfty}, we study the behaviour of the function
\begin{equation}\label{f:fl}
f_\Lambda(r) := J_\Lambda(v_r) - J_\Lambda(1)\, , \qquad r\in (0, \inradius]\, , 
\end{equation}
where $v_r$ is defined by \eqref{f:ur},
and we analyze in particular the value of 
\begin{equation}\label{f:muL} 
\mu_\Lambda  := \min _{ (0, \inradius] } f _\Lambda (r) \,.
\end{equation} 

\medskip

\begin{lemma}\label{l:fl}
There exist two values
$0<\Lbda < \Lbdb$, with
\begin{equation}\label{f:Lbdb}
\Lbdb
:= 
\frac{1}{\r |\Omega_{\r}|}
\,,
\end{equation} 
being $\r$ given by Lemma \ref{l:rinfty}, such that:
\smallskip
\begin{itemize}
\item[(a)]
For every $\Lambda\in (0, \Lbda]$, the function
$f_\Lambda$ is strictly monotone decreasing
(and, in particular, $\mu_\Lambda  =  f_\Lambda(\inradius) = 1/\inradius > 0$).

\smallskip
\item[(b)]
For every $\Lambda\geq \Lbda$, there exist points
$0 < r_\Lambda\leq \rho_\Lambda< \inradius$ such that:
\begin{gather*}
r_{\Lbda} = \rho_{\Lbda},\qquad
\Lambda\mapsto r_\Lambda\
\text{strictly decreasing},
\quad
\Lambda\mapsto \rho_\Lambda\
\text{strictly increasing},
\\
f_\Lbda(r_\Lambda) = f_\Lbda(\rho_\Lambda) = 0,
\qquad
f_\Lbda(r) > 0
\ \Longleftrightarrow \ r \in (r_\Lambda, \rho_\Lambda).
\end{gather*}
In particular, for $\Lambda = \Lbda$, the map $f _\Lambda$ admits a  flex at $r_{\Lbda} = \rho_{\Lbda}$, whereas, 
for every $\Lambda > \Lbda  $,  
$f _\Lambda$ admits a local minimum at $r_\Lambda$ and a local maximum at $\rho _\Lambda$.

\smallskip
\item[(c)]  For every $\Lambda \in (0,  \Lbdb)$, it holds 
$ \mu_\Lambda   >0$ (and $\mu_\Lambda  = \min  \{ 1/ \inradius, f _\L ( r _\L) \}$).

\smallskip
\item[(d)] For every $\Lambda \geq   \Lbdb$, it holds 
$\mu_\Lambda  = f _\L ( r _\L)\leq 0$, 
and $\mu_\Lambda  = 0$ if and only if $\L = \Lbdb$.

\end{itemize}
\end{lemma}

\begin{remark} By inspection of the proof of Lemma \ref{l:fl} given hereafter, it turns out that, for every $\L \geq \Lbda$, the radius $r _\L$ 
can be identified as stated in Theorem \ref{t:Linfty}, namely as 
the smallest root in $(0, \inradius]$ of the equation
$$\frac{1}{r^2} = \Lambda  |\partial\Omega_r|\,.$$
\end{remark}

{\it Proof of Lemma \ref{l:fl}}. 
A direct computation shows that
\[
J_\Lambda(v_r) = \Lip{v_r} + \Lambda(|\{v_r > 0\}|)
= \frac{1}{r} + \Lambda (|\Omega| - |\Omega_r|),
\qquad r\in (0, \inradius], 
\]
hence 
\[
f_\Lambda(r)
=  \frac{1}{r} - \Lambda  |\Omega_r|\,,
\quad
f_\Lambda'(r)
=  -\frac{1}{r^2} + \Lambda  |\partial\Omega_r|\,,
\qquad r\in (0, \inradius]\,.
\]
We have that
\[
f_\Lambda'(r) > 0
\ \Longleftrightarrow \
\sqrt[n-1]{\Lambda|\partial\Omega_r|}
> r^{-2/(n-1)}.
\]
Since, by the Brunn--Minkowski theorem, the map
$r \mapsto \sqrt[n-1]{|\partial\Omega_r|}$ is decreasing and concave in $[0, \inradius]$,
whereas $r \mapsto r^{-2/(n-1)}$ is decreasing and convex in $(0, \inradius]$,
there exists a unique $\Lambda' > 0$
such that (a) and (b) hold.

Observe that $f_0(r) = 1/r$ is positive and monotone decreasing in $(0, \inradius]$, 
whereas for every $r\in (0, \inradius)$ the map $\Lambda \mapsto f_\Lambda(r)$ is
affine with strictly negative slope. 
By continuity, there exists a value $\Lbdb > \Lbda$ such that
$$\min_{r\in (0, \inradius]} f_\Lambda(r) > 0
\ \ \forall \Lambda < \Lbdb\, , 
\qquad
\min_{r\in (0, \inradius]} f_{\Lbdb}(r) = 0,
\qquad
\min_{r\in (0, \inradius]} f_\Lambda(r) < 0
\ \ \forall \Lambda > \Lbdb.
$$
 
Moreover, if $\r \in \argmin f_{\Lambda^*}$,
then the pair $(\r, \Lbdb)$
satisfies the conditions
\[
f_\Lambda(r) = 0,
\quad
f'_\Lambda(r) = 0,
\]
i.e.
\begin{equation}\label{f:lambda}
\frac{1}{r} - \Lambda|\Omega_r| = 0,
\quad
-\frac{1}{r^2} + \Lambda|\partial\Omega_r| = 0.
\end{equation}
From the first equation we have $\Lambda = (r |\Omega_r|)^{-1}$; substituting into the second equation
we get the condition
\begin{equation}\label{f:psi}
\frac{|\partial\Omega_r|}{|\Omega_r|}= \frac{1}{r}\,.
\end{equation}
By Lemma~\ref{l:rinfty},
there exists a unique $\r\in (0, \inradius)$ satisfying \eqref{f:psi},
so that $\Lbdb= (\r |\Omega_{\r}|)^{-1}$.

The properties stated in (c) and (d) follow.
\qed

\bigskip
\begin{example} 
When $\Omega = B_R \subset \R^n$, for $r\in(0,R]$ we have:
\[
f_\Lambda(r) = \frac{1}{r} - \kappa_n \Lambda (R - r)^n,
\quad
f_\Lambda'(r) = -\frac{1}{r^2} + n \kappa_n \Lambda (R- r)^{n-1}.
\]
An explicit computation gives 
\[
\Lbda = \frac{1}{4} \left(\frac{n}{n-1}\right)^{n-1}\Lbdb \,.
\]
where the value of $\Lambda ^*$ has been already computed in Example \ref{e:ball}. 

The graph of $f_\Lambda$ in the case $n=2$ and $R=1$, for the three choices  $\Lambda = \Lbda\simeq 1.07$, 
$\Lambda=\Lbdb \simeq 2.15$,
and $\Lambda = 3$, is shown in Figure~\ref{fig:ball}.
\begin{figure}
	\includegraphics[width=8cm]{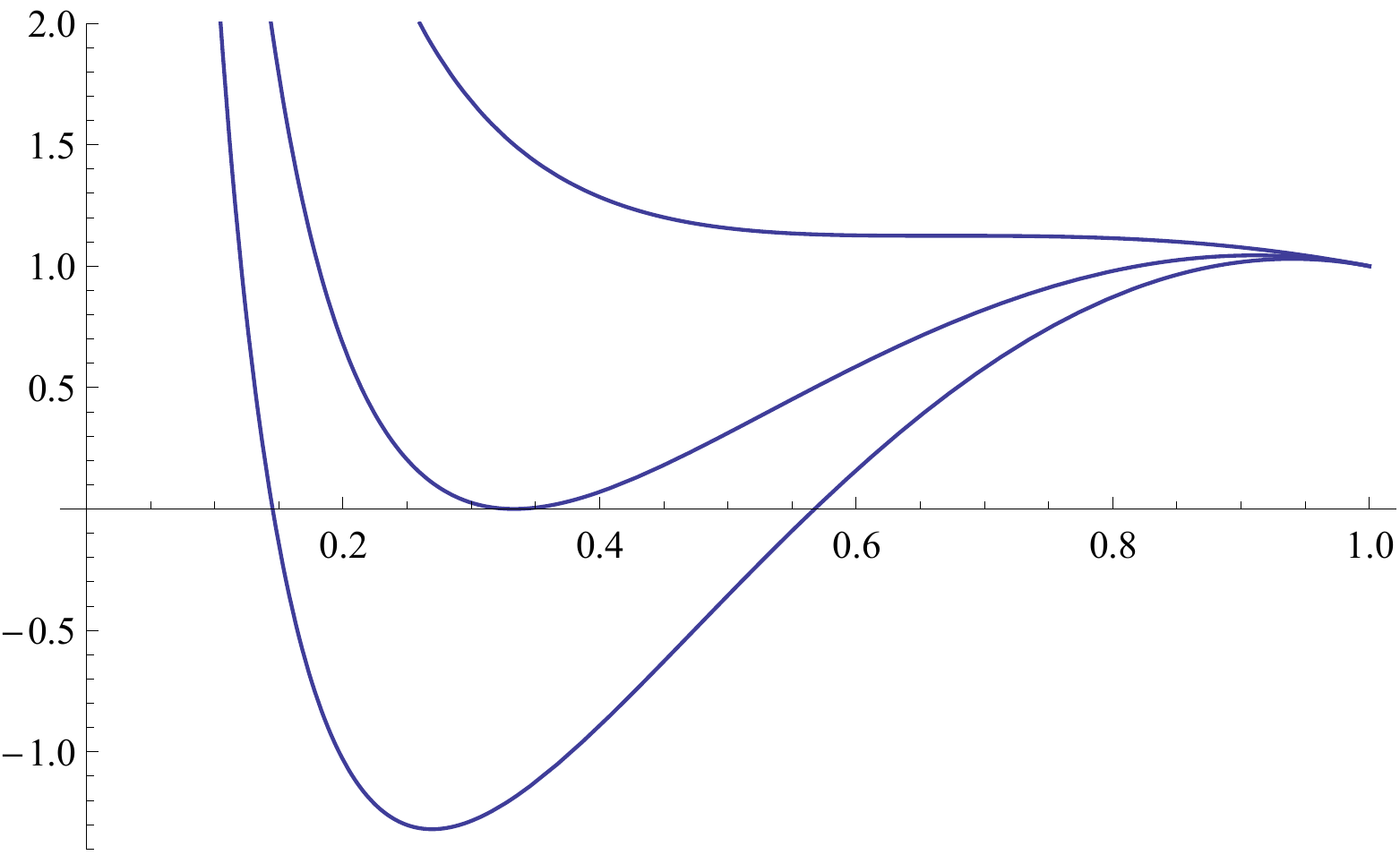}
	\label{fig:ball}
	\caption{ Plot of the map $f_\Lambda$ when $\Omega$ is the unit two--dimensional ball, for $\Lambda = \Lbda\simeq 1.07$ (top), $\Lambda=\Lbdb\simeq 2.15$ (center),
		and $\Lambda = 3$ (bottom)}
\end{figure}
\end{example}

\bigskip
\begin{proof}[Proof of Theorem \ref{t:Linfty}] 
The existence of a unique value $r ^*$ satisfying \eqref{f:r*} has been already proved in Lemma \ref{l:rinfty}. 

Let now $\Lambda >0$. 
By Lemma~\ref{l:existence}, the functional $J_\Lambda$ admits
a minimizer $v \in \llipst_1(\Omega)$.

Since $v = 1$ on $\partial\Omega$ and
$|v(x) - v(y)| \leq \Lip{v}\, |x-y|$ for every 
$x,y\in\overline{\Omega}$,
we observe that 
\begin{equation}\label{f:a}
\{v > 0\} \supseteq D_r, \hbox{ with } r := 1/\Lip{v} \,.
\end{equation}
In particular, 
\begin{equation}\label{f:b}
\text{ if } \Lip{v} < 1/\inradius, \text{ then }\{v > 0\} = \Omega \,.
\end{equation}
Then we distinguish two cases.

\smallskip
If $\Lip{v} < 1/\inradius$, 
then by \eqref{f:a} we conclude that necessarily $v\equiv 1$
and $\Lambda \leq \Lbdb$, for otherwise
$J_\Lambda(v_{r_\Lambda}) < J_\Lambda(1)$
by Lemma~\ref{l:fl}.

\smallskip
If $\Lip{v} \geq 1/\inradius$, let $\overline r := 1/\Lip{v} \in (0, \inradius]$.
The function $v_{\overline r} $, defined in \eqref{f:ur}, has the same Lipschitz constant as $v$.
Moreover, since $\{v_{\overline r} > 0\} = D_{\overline r}$,  by \eqref{f:b} we deduce that
$J_\Lambda(v_{\overline r}) \leq J_\Lambda(v)$.
Since $v$ is a minimizer, then equality must hold and hence
$\{v > 0\} = D_{\overline r}$.

Then, recalling the definition \eqref{f:fl} of the function $f _\Lambda$,
we have 
\[
\min_{(0,\inradius]} f_\Lambda
\leq f_\Lambda(\overline r) = J_\Lambda(v_{\overline r}) - J_\Lambda(1)
=  J_\Lambda(v) - J_\Lambda(1) \leq \min_{(0,\inradius]} f_\Lambda  \,.
\]
Hence, we have $\min_{(0,\inradius]} f_\Lambda= J_\Lambda(v) - J_\Lambda(1) \leq 0$. 
 By Lemma~\ref{l:fl}, this implies that $\Lambda \geq \Lambda ^*$, and $\overline r = r _\Lambda$. 

From the above analysis, it follows that $\Lambda_\infty  = \Lbdb$,
and that, for every $\Lambda \geq \Lbdb$, $v_{r_\Lambda}$  is a non-constant solution to $(P)_\Lambda$  (and any other non-contant solution has the same Lipschitz constant and the same positivity set as $v_{r_\Lambda}$). 
\end{proof}

\bigskip

\begin{proof}[Proof of Theorem \ref{t:isoper}] 
By Theorem \ref{t:Linfty}, Lemma \ref{l:rinfty},  and the explicit computation in Example~\ref{e:ball},
we have to prove that
\[
\max_{r\in [0, \inradius]} r |\Omega_r|
\leq
\frac{\kappa_n n^n R^{n+1}}{(n+1)^{n+1}},
\qquad
\text{with}\
R = \left(\frac{|\Omega|}{\kappa_n}\right)^{1/n}\,.
\]
By the Brunn--Minkowski inequality,
the function $\gamma(r) := |\Omega_r|^{1/n}$ is concave in $[0, \inradius]$,
hence
$\gamma(r) \leq \gamma(0) + r\, \gamma'_+(0)$, i.e.
\[
|\Omega_r|^{1/n} \leq |\Omega|^{1/n} - \frac{r}{n}\, |\Omega|^{-1+1/n} |\partial\Omega|\,,
\qquad\forall r\in [0, \inradius],
\]
or
\begin{equation}\label{f:fi1}
|\Omega_r| \leq |\Omega| 
\left(1 - \frac{r}{n}\, \frac{|\partial\Omega|}{|\Omega|}\right)^n\,,
\qquad\forall r\in [0, \inradius].
\end{equation}
(For related inequalities see \cite[\S~3]{Cg}.)
Hence,
\begin{equation}\label{f:fi}
	r |\Omega_r| \leq r |\Omega| 
	\left(1 - \frac{r}{n}\, \frac{|\partial\Omega|}{|\Omega|}\right)^n
	=: \varphi(r)\,.
	\qquad\forall r\in [0, \inradius].
\end{equation}

It is easy to check that $\varphi$ attains its maximum at
$r_0 := \frac{n}{n+1} \frac{ |\Omega|}{ |\partial\Omega|}$,
hence
\[
\max_{[0, \inradius]} \varphi
= \varphi(r_0) = \frac{n^{n+1}}{(n+1)^{n+1}} \cdot \frac{|\Omega|^2}{|\partial\Omega|}\,.
\]
By the isoperimetric inequality and the definition of $R$ we have that
\[
\frac{|\Omega|^2}{|\partial\Omega|}
= |\Omega|^{(n+1)/n} \frac{|\Omega|^{(n-1)/n}}{|\partial\Omega|}
	\leq \kappa_n^{(n+1)/n} R^{n+1} \frac{1}{n \kappa_n^{1/n}}
	= \frac{\kappa_n R^{n+1}}{n},
\]
with equality if and only if $\Omega$ is a ball, and finally
\[
\max_{r\in [0, \inradius]} r |\Omega_r|
\leq \varphi(r_0)
\leq
\frac{\kappa_nn^{n} R^{n+1}}{(n+1)^{n+1}}\,,
\]
with equality if and only if $\Omega$ is a ball.
\end{proof}

\bigskip

\begin{proof}[Proof of Theorem \ref{t:uniqP}] 
We are going to prove the following two facts:
\begin{itemize}
\item[(a)] If $\L \geq \L _\infty$ and $\L < \L _{sing}$, there are at least two distinct solutions. 
\item[(b)] If $\L > \L _\infty$ and $\L \geq \L _{sing}$, there is a unique solution. 
\end{itemize} 
Let us check first that the statement easily follows from (a) and (b). 

{\it Case $\L_{sing} \leq\L _\infty$}.  If there is a unique solution, it must be $\L > \L _\infty$  (otherwise both $v_{r^*}$ and the constant function $1$ are solutions). Viceversa, if $\L > \L _\infty$, by (b) there is a unique solution. 

{\it Case $\L_{sing} >\L _\infty$}.  If there is a unique solution, it must be $\L \geq \L _{sing}$ (otherwise by (a) there would be at least two solutions). Viceversa, if $\L \geq \L _{sing}$, by (b) there is a unique solution. 

Let us now prove (a). Let $\L \geq \L _\infty$  and $\L < \L _{sing}$. Then, recalling from Lemma \ref{l:fl} that the map $\Lambda \mapsto r _\Lambda$ is monotone decreasing, we infer that
$r_\Lambda > r_{sing}$,
so that $v_{r_\Lambda}$ is not everywhere differentiable in $D_{r_\Lambda}$. On the other hand, 
the function $w_{r_\L}$ defined by \eqref{f:pot}, which by  Corollary \ref{c:corharm}  
is a minimizer of $J_\Lambda$, is differentiable everywhere in $D_{r_\Lambda}$  (see \cite{EvSm}). Hence
we have at least two different solutions, $v_{r_\Lambda}$ and $w_{r_\L}$.

Finally, let us prove (b).  Let $\L > \L _\infty$ and $\L \geq \L _{sing}$. 
Since  $\L > \L _\infty$, by Theorem \ref{t:Linfty} any solution $v$ satisfies 
$\{v > 0\}\cap\Omega = D_{r_\Lambda}$.
Let us prove that $v = v_{r_\Lambda}$.

Assume by contradiction that there exists a point $x\in D_{r_\Lambda}$ such that
$v(x) \neq v_{r_\Lambda}(x)$.

Since  $r_\Lambda \leq r_{sing}$, 
the distance function $\dist(\cdot, \partial\Omega)$ from the boundary of $\Omega$
is differentiable everywhere in $D_{r_\Lambda}$.
Therefore the point $x$ admits a unique projection $y\in\partial\Omega$
such that $|x-y| = \dist(x, \partial\Omega)$.
Setting $\nu := \frac{x-y}{|x-y|}$, and $y_t := y + t\nu$, we have 
\[
y_t \in D_{r_\Lambda},
\quad\forall t\in (0, r_\Lambda),
\qquad
z := y_{r_\Lambda}\in \partial \Omega_{r_\Lambda}.
\]
Since $v_{r_\Lambda}(z) = v(z) = 0$,
if $v(x) > v_{r_\Lambda}(x)$, then we would have
\begin{equation}\label{f:contra1}
v(x) - v(z) > v_{r_\Lambda}(x) - v_{r_\Lambda}(z) = \frac{|x-z|}{r_\Lambda}\,,
\end{equation}
hence $\Lip{v} > 1/r_\Lambda$, a contradiction.
Similarly, if $v(x) < v_{r_\Lambda}(x)$, 
since $v_{r_\Lambda}(y) = v(y) = 1$, then we would have
\begin{equation}\label{f:contra2} 
v(y) - v(x) > v_{r_\Lambda}(y) - v_{r_\Lambda}(x)
= \frac{|x-y|}{r_\Lambda},
\end{equation}
and again $\Lip{v} > 1/r_\Lambda$, a contradiction.

Notice that  to obtain the last equality in formulas \eqref{f:contra1} and \eqref{f:contra2} we have used the identity 
$\dist(\xi, \partial\Omega) = r_\Lambda - |\xi -z|$ holding for every $\xi$ in the segment $[y, z]$ 
because, for $r _\Lambda \leq r_{sing}$, $\dist(\cdot, \partial\Omega)$ is differentiable  in $D_{r_\Lambda}$. 
\end{proof}

\section{Proofs of the results in Section \ref{s:Bernoulli}}\label{sec:proof2}

Before giving the proofs of Propositions \ref{p:PP} and \ref{p:varsol}, we need to recall a result
from our paper \cite{CF8} about the Bernoulli problem \eqref{f:Pinfty} on convex domains
 (therein, also the more general case of non-convex domains is considered).  

We first resume a few preliminary definitions. 
	A solution $u$ to \eqref{f:Pinfty}  is called non-trivial
	if the set $\{u = 0\}$ has non-empty interior. 
Given $r \in (0, \inradius)$, we define $w_r$ as the infinity harmonic potential of $\overline \Omega _r $, namely the unique solution to
\begin{equation}\label{f:potr} 
\begin{cases}
\Dinf w_r = 0 & \text{ in } D_r:= \Omega \setminus \overline \Omega _r
\\
w_r= 1 & \text{ on } \partial \Omega
\\ 
w_r = 0 & \text{ in }  \overline \Omega_{r}\,.
\end{cases} 
\end{equation}
Finally, still for $r \in (0, \inradius)$, 
we introduce the subset of $D_r$  defined by 
\[
\widehat{D}_r := \bigcup_{y\in\partial\Omega_r}
\left\{
]y,z[:\ z\in\Pi_{\partial\Omega}(y)
\right\}\, , 
\]
where $]y, z[$ denotes the open (i.e., without the endpoints) segment joining $y$ to $z$, and $\Pi _{\partial \Omega} (y)$ denotes the set of projections of $y$ onto $\partial \Omega$.

\begin{theorem}[See \cite{CF8}]\label{t:B}
	(a) For every $\lambda >  \frac{1}{\inradius}$,
	%$\lambda  \in \big [0, \frac{1}{\inradius} \big ) $,  
	the function $w_{\frac{1}{\lambda}}$ is the unique non-trivial solution to
	problem \eqref{f:Pinfty} ; moreover it satisfies the estimates
	\begin{equation}\label{f:cfdist}
		1- \lambda \dist (x, \partial \Omega) \leq  w_{\frac{1}{\lambda}} (x) \leq \lambda \dist (x, \partial \Omega _{\frac{1}{\lambda}} )  \quad \text{ in }  \overline D_{\frac{1}{\lambda}} \,, 
		\text{ with equalities in } \widehat D_{\frac{1}{\lambda}}\,.
	\end{equation}

	(b) For every $\lambda \in \big( 0, \frac{1}{\inradius}  \big ]$,
	problem \eqref{f:Pinfty}  does not admit non-trivial solutions.
\end{theorem}

We can now give: 

\bigskip

\begin{proof}[Proof of Proposition \ref{p:PP}]  
From Corollary \ref{c:corharm} we know that, for every $\L \geq \L _\infty$, the infinity harmonic potential $w _{r_\L}$ of $D _{r_\L}$ is a solution to problem  \ref{f:P}. Moreover, by Theorem  \ref{t:B}, the function $w _{r_\L}$ solves problem  \eqref{f:Pinfty}   (for  $\l =   \frac{1}{r_\L}$). 

In order to prove that there are no other solutions to \ref{f:P} which solve \eqref{f:Pinfty}, it is enough to recall that 
the positivity set of any solution to problem \ref{f:P} is given by $D_{r_\L}$
({\it cf.}\ \eqref{f:split} in Theorem \ref{t:Linfty}); since clearly  a solution to \ref{f:P} needs to be infinity harmonic in its positivity set in order to solve \eqref{f:Pinfty}, it agrees necessarily with $w_{r_\L}$.  

Finally,  since we know respectively 
from  Theorem \ref{t:Linfty} and from Corollary \ref{c:corharm}   that $v_{r_\L}$ and
$w _{r _\L}$ are both solutions to \ref{f:P}, in case of uniqueness we conclude
that  $v_{r_\L} = w _{r _\L}$; moreover, from the first part of the statement already proved we deduce that such function
solves \eqref{f:Pinfty}  (for  $\l =   \frac{1}{r_\L}$). 
\end{proof}

\bigskip
\begin{proof}[Proof of Proposition \ref{p:varsol}]
By  Proposition \ref{p:PP}, 
among  solutions to problem \ref{f:P}  
there is one  which solves problem \eqref{f:Pinfty}  (for $\l =   \frac{1}{r_\L}$) if and only if  
$
\L \geq \L _\infty$ (and in this case it is given precisely by  
the function $w _{r_\L}$).
Then it is enough to observe that, by the continuity and decreasing monotonicity of the map $\mathcal R: \L \mapsto r_\L$, condition $\L \geq \L _\infty$ is equivalent to \eqref{f:cond1} (being $\l =   \frac{1}{r_\L}$).  

The last part of the statement follows from Theorem \ref{t:B}, combined with the fact that it holds
$\widehat D _r = D _r$ if and only if $r \leq r _{sing}$.
\end{proof}

\section{Proofs of the results in Section \ref{s:approx}.}\label{sec:proof3} 

Let $\Lambda \geq \Lambda _\infty$. As explained in Section \ref{s:approx}, our first step towards the proof of Theorem  \ref{t:convmin} is the study of the asymptotics of 
a sequence of solutions to problem \eqref{f:minp}, for a fixed $\lambda$, in the limit as $p = p _j \to + \infty$.

To that aim, it is useful to notice preliminarily that
the minimum of the functionals $\JLia$ defined in \eqref{f:double} can be explicitly computed as 
\begin{equation}\label{f:minvalue} 
\min_{\llipst_1(\Omega)}\JLia = 
\begin{cases}
\lambda + \Lambda |\Omega|
&\text{  if}\ \lambda \in [0, 1/\inradius),\\
\smallskip
\lambda + \Lambda |D_{1/\lambda}|
&\text{  if}\ \lambda \geq 1/\inradius.
\end{cases}
\end{equation} 
Indeed, if $\lambda < 1/\inradius$,
then $\{u > 0\} = \overline{\Omega}$ for every $u\in\llipst_1(\Omega)$ with
$\Lip{u} \leq \lambda$;
on the other hand, 
for $\lambda \geq 1/\inradius$,
by arguing as in the proof of Theorem \ref{t:Linfty} we see that a minimizer is given 
by the function $v_{1/\lambda} = [1- \lambda\, \dist(\cdot, \partial\Omega)]_+$. 

\begin{figure}
	\includegraphics[width=8cm]{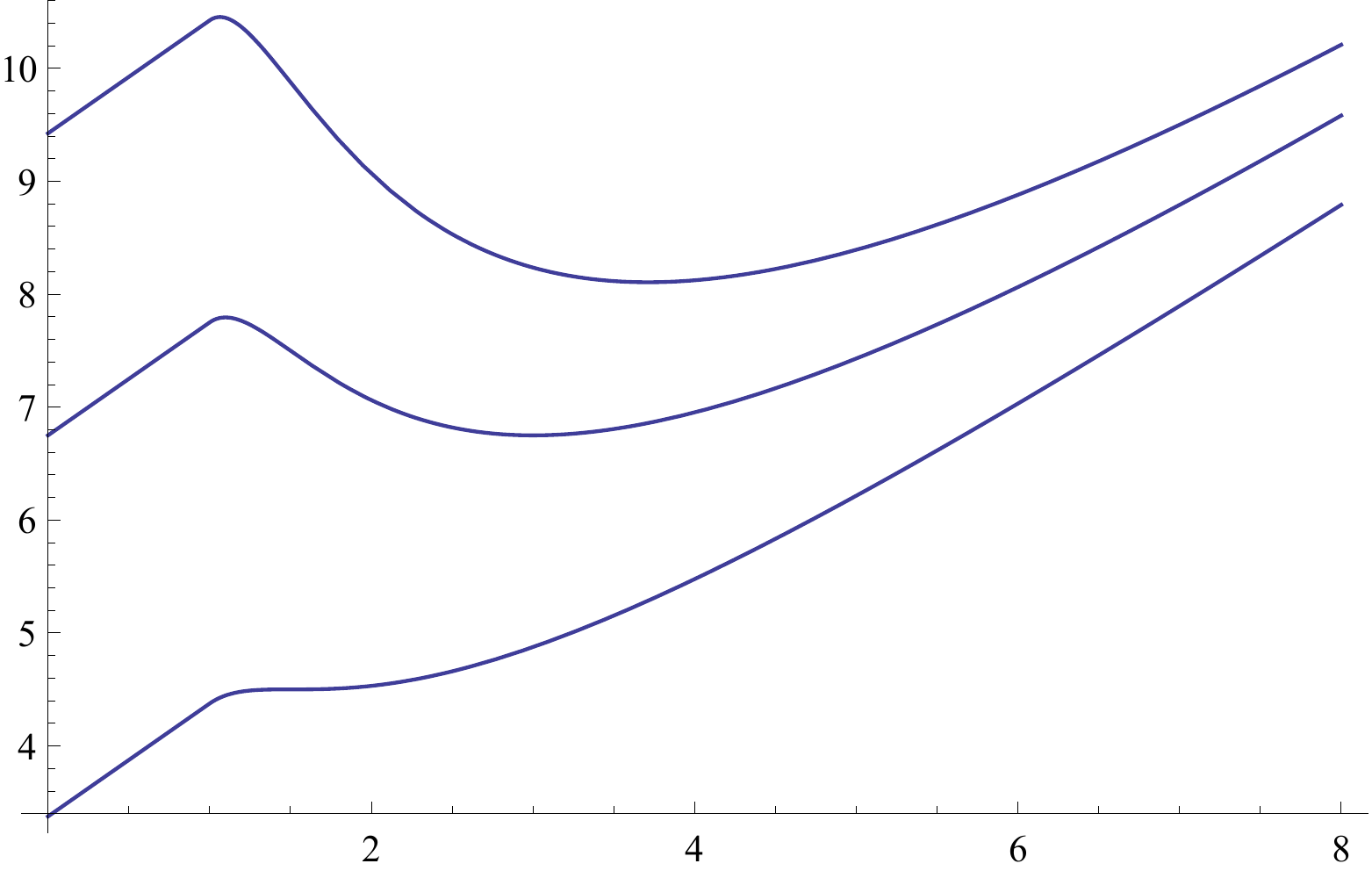}
	\caption{Plot of the map $\lambda\mapsto \min \JLia$  in \eqref{f:minvalue} when $\Omega$ is the unit two--dimensional ball, for $\Lambda = \Lbda\simeq 1.07$ (top), $\Lambda=\Lbdb\simeq 2.15$ (center),
		and $\Lambda = 3$ (bottom)}
	\label{fig:jball}
\end{figure}

Figure \ref{fig:jball} represents the plot of the map $\lambda\mapsto \min \JLia$ when $\Omega$ is  the unit two--dimensional ball, 
for three different values of $\Lambda$.

\begin{lemma}[Convergence of minimizers at fixed $\lambda$]\label{l:conv}
Let $\L \geq \L _\infty$. Let $(u^{p_j, \lambda})_j$ be a sequence of solutions to problem \eqref{f:minp},  for a given $\lambda>0$ and $p = p _j \to + \infty$.  Then, up to passing to a (not relabeled) subsequence, we have 
\begin{equation}\label{f:lca}
u^{p_j, \lambda} \rightharpoonup u^\lambda\
\text{weakly in}\ W^{1,q}(\Omega) \quad \forall q>1\, , 
\qquad
u^{p_j, \lambda} \to u^\lambda,\
\text{uniformly in}\ \overline{\Omega}\,,
\end{equation}
where $u ^ \lambda$ is a solution to problem \eqref{f:mina} and is infinity harmonic in its positivity set. 

Furthermore:
\smallskip
\begin{itemize}
\item[--] if $\lambda\in [0, 1/\inradius)$, then
$u^\lambda\equiv 1$ and
$\upja \equiv 1$ for $j$ large enough;
\smallskip
\item[--] 
 if $\lambda \geq 1/\inradius$, then
\begin{equation}\label{f:lcb}
\{u^\lambda > 0 \} = D_{1/\lambda}\cup\partial\Omega,\quad
\big \| \nabla u^{\lambda} \big  \| _{\infty} =  \lambda\,.
\end{equation}
\end{itemize} 
\end{lemma}

We prepone to the proof of Lemma \ref{l:conv} a useful observation:

\begin{remark}\label{r:mon} 
For every $u \in W ^{1, p} _ 1 (\Omega)$ and every fixed $\lambda$, the map
$p \mapsto J ^ {p, \lambda} _\Lambda (u)$ is monotone non-decreasing. We omit the proof of this property, which can be found in \cite[Proposition~1]{KawSha}. 
Moreover, it is readily  checked that the limit as $p \to + \infty$ of $J ^ {p, \lambda }  _\Lambda (u)$ is given precisely by 
the number $J ^\lambda _\Lambda (u) \in [0,+\infty]$ defined in \eqref{f:double}.
\end{remark}

\begin{proof}[Proof of Lemma \ref{l:conv}]  
Since
\[
\frac{1}{p}\int_\Omega \left(\frac{|\nabla \upa|}{\lambda}\right)^p\, dx + \lambda
\leq \JLpa(\upa)
\leq \JLpa(1) = \lambda + \frac{p-1}{p} \Lambda |\Omega|
\leq \lambda + \Lambda |\Omega|,
\]
we get 
\[
\|\nabla \upa\|_p \leq
\lambda  (p \Lambda |\Omega|)^{\frac{1}{p}}.
\]
For every fixed exponent $q\in (1, +\infty)$,  by H\"older inequality, 
for every $p > q$  it holds
\begin{equation}\label{f:estigupa}
\|\nabla\upa\|_q \leq \|\nabla\upa\|_p \, |\Omega|^{\frac{p-q}{pq}}
\leq 
\lambda  (p \Lambda |\Omega|)^{\frac{1}{p}} |\Omega|^{\frac{p-q}{pq}} = \lambda  (p \Lambda)^{\frac{1}{p}} |\Omega|^{\frac{1}{q}} 
\leq C,
\end{equation}
where $C>0$ is a constant independent of $p$.

Therefore, the family $(\upa)_p$ is uniformly bounded in $W^{1,q}(\Omega)$,
for every $q > 1$. 
Using a diagonal argument, we can construct an increasing sequence $p_j\to +\infty$
satisfying \eqref{f:lca}, for some $u^\lambda$.
Moreover, from \eqref{f:estigupa},
we deduce that
$\|\nabla u^\lambda\|_\infty \leq \lambda$.
Since $u^\lambda = 1$ on $\partial\Omega$, we conclude
that $u^\lambda\in \llipst_1(\Omega)$ and $\Lip{u^\lambda}\leq \lambda$.

\smallskip
The fact that $u ^\lambda$ is $\infty$-harmonic in its positivity set is a standard consequence of the fact that the functions
$u ^ {p_j,\lambda}$ are $p_j$-harmonic in their positivity set,  with  $p_j \to + \infty$, see for instance the arguments in \cite[proof of Theorem 1]{RossiTeix}. 

\smallskip
Let us prove that $u^\lambda$ is a minimizer of $\JLia$ in $\llipst_1(\Omega)$.
Let us fix $\varepsilon > 0$.
Since $\upja\to u^\lambda$ uniformly in $\overline{\Omega}$,
there exists an index $j_\varepsilon\in\N$ such that
\[
|\{u^\lambda > 0\}| \leq  |\{\upja > 0\}| + \varepsilon,
\qquad\forall j > j_\varepsilon\, ;
\]
moreover, for every $j> j_\varepsilon$, we have
\[
\JLpja(\upja) \geq \lambda + \frac{p_j-1}{p_j} \Lambda |\{\upja> 0\}|
\,.
\]
So we obtain
\[
\begin{split}
\JLpja(u^\lambda)
& \leq
\frac{1}{p_j}|\{u^\lambda > 0\}| + \lambda + \frac{p_j-1}{p_j} \Lambda |\{u^{\lambda} > 0\}|
\\ & \leq
\frac{1}{p_j} |\Omega |  + \lambda 
+ \frac{p_j-1}{p_j} \Lambda (|\{\upja > 0\}|   +  \varepsilon) 
\\ &  \leq
\frac{1}{p_j} |\Omega |   
+ \JLpja(\upja) + C\varepsilon. 
\end{split}
\]
For every $u\in\llipst_1(\Omega)$, passing to the limit as $j \to + \infty$ ({\it cf.} Remark \ref{r:mon}), we obtain
\[
\JLia(u^\lambda) = \lim_{j\to +\infty} \JLpja(u^\lambda) \leq
\liminf_{j\to+\infty} \JLpja(\upja)
\leq
\liminf_{j\to+\infty} \JLpja(u)
= \JLia(u),
\]
so that $u^\lambda$ is a minimizer of $\JLia$ in $\llipst_1(\Omega)$.

\smallskip
If $\lambda\in[0,1/\inradius)$,
then $u^\lambda \equiv 1$ (because $u ^ \lambda$ is $\infty$-harmonic in $\Omega$, with $u ^ \lambda = 1$ on $\partial \Omega$). 
Since $\upja \to 1$ uniformly, for $j$ large enough
then $\{\upja > 0\} = \overline{\Omega}$, hence $\upja \equiv 1$.

\smallskip
It remains to prove that, for $\lambda\geq 1/\inradius$, the positivity set $\{u^\lambda > 0 \}$ coincides with $D_{1/\lambda}\cup\partial\Omega$, and $\Lip{ u^{\lambda} } =  \lambda$. 
Since $\Lip{u ^\lambda}\leq \lambda$, we have  
$\{u^\lambda > 0 \} \supseteq D_{1/\lambda}\cup\partial\Omega$.
On the other hand, since $u^\lambda$ is a minimizer of $\JLia$ in $\llipst_1(\Omega)$, it minimizes
$|\{u>0\}|$ among all functions $u\in\llipst_1(\Omega)$
with $\Lip{u} \leq \lambda$, which implies (by taking the competitor $v _{1/\lambda}$)  that $|\{u^\lambda > 0\}|  \leq |D_{1/\lambda}|$. We infer that $\{u^\lambda > 0\}  = D_{1/\lambda}\cup\partial\Omega$.

Finally, as a consequence of \eqref{f:minvalue} and the equality 
$\{u^\lambda > 0\}  = D_{1/\lambda}\cup\partial\Omega$, we obtain that $\Lip{ u ^ \lambda } = \lambda$. 
\end{proof}

\bigskip
\begin{proof}[Proof of Theorem \ref{t:convmin}] 
Observe that the limit as $p\to +\infty$ in \eqref{f:min}
does exist,  thanks to Remark \ref{r:mon}. 
Let $u ^ {p, \lambda}$ and $u ^\lambda$ be as in Lemma \ref{l:conv}. 

For every $\lambda \geq 1/\inradius$ let
$v_{1/\lambda}(x) := \left [ 1-\lambda \dist(x,\partial\Omega) \right ]_+$, $x\in\overline{\Omega}$.
Clearly
\[
\{v_{1/\lambda} > 0\}\cap\Omega = D_{1/\lambda} = \{u^\lambda > 0\}\cap\Omega,
\qquad
|\nabla v_{1/\lambda}| = \lambda\
\text{a.e.\ in}\ D_{1/\lambda},
\]
so that $\JLpa(u^\lambda) \leq \JLpa(v_{1/\lambda})$.
Observe that
\begin{equation}\label{f:estiva}
\begin{split}
\JLpa(v_{1/\lambda}) & =
\frac{1}{p}|D_{1/\lambda}| + \lambda + \frac{p-1}{p} \Lambda |D_{1/\lambda}|
\leq
\frac{1}{p}|\Omega| + \lambda + \Lambda |D_{1/\lambda}|\,.
\end{split}
\end{equation}
Hence,
\begin{equation}\label{f:upbo}
\inf_{\lambda \geq 1/\inradius} \JLpa(\upa) \leq
\inf_{\lambda \geq 1/\inradius} \JLpa(u^\lambda) \leq 
\inf_{\lambda \geq 1/\inradius} \JLpa(v_{1/\lambda})
\leq
\frac{1}{p} 
|\Omega| + \frac{1}{r_\Lambda} + \Lambda |D_{r_\Lambda}| \,.
\end{equation}

Notice carefully that in the last inequality above we have exploited the assumption $\L \geq \L _\infty$, and we have used Lemma \ref{l:fl} (in particular, the definition of $r _\L$ given therein, and part (d) of the statement). 

In view of the upper bound obtained in \eqref{f:upbo}, and taking into account 
that $\JLpa(\upa) > \lambda$ for every $\lambda > 0$,
we see that the infimum w.r.t.\ $\lambda$ in \eqref{f:min}
can be taken for
$\lambda \in I = [1/\inradius, C_\Lambda]$, being $C_\Lambda := |\Omega| + \frac{1}{r_\Lambda} + \Lambda |D_{r_\Lambda}|$. 

From the explicit form of $\JLpa(v_{1/\lambda})$ in \eqref{f:estiva}
we also get
\[
\JLpa(v_{1/\lambda}) \geq 
\frac{p-1}{p}\left(\lambda + \Lambda |D_{1/\lambda}|\right)
\geq
\frac{p-1}{p}\left(\frac{1}{r_\Lambda} + \Lambda |D_{r_\Lambda}|\right),
\]
that, together with \eqref{f:estiva}-\eqref{f:upbo}, gives
\begin{equation}\label{f:limva}
\lim_{p\to +\infty}
\inf_{\lambda \geq 1/\inradius}
\JLpa(v_{1/\lambda})
=
\frac{1}{r_\Lambda} + \Lambda | D_{r_\Lambda}|.
\end{equation}

Let $(\lambda_k)\subset I$ be an enumeration of the rational points in $I$.
Since the map $\lambda\mapsto\JLpa(u)$ is continuous, it is easy to check that
\[
\inf\{\JLpa(\upa):\ \lambda \geq 1/\inradius\} =
\inf\{J_\Lambda^{p, \lambda_k}:\ k\in\N\},
\qquad\forall p>1.
\]

Using Lemma~\ref{l:conv} and a diagonal argument,
we can construct a sequence $p_j\to +\infty$ such that,
for every $q>1$, for $j\to +\infty$,
\begin{gather*}
u^{p_j, \lambda_k} \rightharpoonup u^{\lambda_k}\
\text{weakly in}\ W^{1,q}(\Omega),
\quad
u^{p_j, \lambda_k} \to u^{\lambda_k},\
\text{uniformly in}\ \overline{\Omega},
\qquad
\forall k\in\N.
\end{gather*}

For every $j\in\N$ let us choose $k_j\in\N$ such that
\begin{equation}\label{f:infa}
\inf\{\JLpja(\upja):\ \lambda \geq 1/\inradius\} \leq 
J_\Lambda^{p_j, \lambda_{k_j}}(u^{p_j, \lambda_{k_j}}) + \frac{1}{p_j}\,.
\end{equation}
Upon extracting a further subsequence (not relabelled),
we can assume that $\lambda_{k_j} \to \overline{\lambda} \in I$, and that
(again using Lemma~\ref{l:conv})
\[
u^{p_j, \overline{\lambda}} \rightharpoonup u^{\overline{\lambda}}\
\text{weakly in}\ W^{1,q}(\Omega),
\quad
u^{p_j, \overline{\lambda}} \to u^{\overline{\lambda}},\
\text{uniformly in}\ \overline{\Omega}.
\]

\medskip
\textbf{Claim:}
For every $\varepsilon > 0$ there exist $j_\varepsilon\in\N$ such that
\begin{equation}\label{f:estuv}
|\{u^{\lambda_{k_j}} > 0\}| = |D_{1/\lambda_{k_j}}| \leq
|\{u^{p_j, \lambda_{k_j}} > 0\}| + \varepsilon,
\qquad
\forall j\geq j_\varepsilon.
\end{equation}

\textsl{Proof of the claim.}
Let $j_0\in\N$ be such that
\begin{equation}\label{f:eskja}
|D_{1/\lambda_{k_j}}| \leq |D_{1/\overline{\lambda}}| + \frac{\varepsilon}{2}\,,
\qquad \forall j\geq j_0.
\end{equation}
Since $u^{p_j, \overline{\lambda}} \to u^{\overline{\lambda}}$
uniformly in $\overline{\Omega}$, there exists $j_\varepsilon\in\N$, $j_\varepsilon \geq j_0$, such that
\begin{equation}\label{f:eskjb}
|D_{1/\overline{\lambda}}| = |\{u^{\overline{\lambda}} > 0\}|
\leq |\{u^{p_j, \overline{\lambda}} > 0\}|+ \frac{\varepsilon}{2}\,,
\qquad \forall j\geq j_\varepsilon.
\end{equation}
From \eqref{f:eskja} and \eqref{f:eskjb} the claim follows.

\medskip 
Let us fix $\varepsilon>0$. 
Using the claim, for every 
$(p, \lambda) = (p_j, \lambda_{k_j})$ with
$j\geq j_\varepsilon$, 
we have that
\[
\JLpa(\upa) \geq \lambda + \frac{p-1}{p}\Lambda \left(|D_{1/\lambda}| - \varepsilon\right),
\]
hence
\[
\begin{split}
\JLpa(\upa) & \leq \JLpa(u^\lambda) \leq \JLpa(v_{1/\lambda})
= 
\frac{1}{p}|D_{1/\lambda}| + \lambda + \frac{p-1}{p} \Lambda |D_{1/\lambda}|
\\ & \leq
\frac{1}{p}|\Omega|
+ \frac{p-1}{p}\Lambda \varepsilon +
\JLpa(\upa).
\end{split}
\]
Finally, from \eqref{f:limva} and \eqref{f:infa} we conclude that \eqref{f:min} holds.
\end{proof}

\bigskip 
{\bf Acknowledgments.} 
We thank Marcello Ponsiglione for some useful discussions.

The authors have been supported by the Gruppo Nazionale per l'Analisi Matematica, 
la Probabilit\`a e le loro Applicazioni (GNAMPA) of the Istituto Nazionale di Alta Matematica (INdAM).

\def\cprime{$'$}

\end{document}